\documentclass{amsart}
\usepackage[margin=1.2in]{geometry}
\usepackage{amssymb}
\usepackage{amscd}
\usepackage[all]{xy}
\usepackage{bbm}
\usepackage{mathrsfs}
\usepackage{enumerate}
\usepackage{tikz}
\usepackage{stmaryrd}
\usepackage{etoolbox}
\usepackage{array}
\usepackage{verbatim}
\usepackage{hyperref}
\usepackage{stmaryrd}
\usepackage{bbding}

\newcommand{\R}{\mathbb{R}}

\newcommand{\eps}{\varepsilon}

\newcommand{\del}{\nabla}

\newcommand{\lap}{\Delta}

\newcommand{\tr}{\operatorname{tr}}

\newcommand{\la}{\langle}
\newcommand{\ra}{\rangle}

\renewcommand{\div}{\operatorname{div}}

\newcommand{\grad}{\del}

\newcommand{\an}{\operatorname{An}}

\theoremstyle{plain}
\newtheorem{theorem}{Theorem}
\newtheorem{corollary}[theorem]{Corollary}
\newtheorem{prop}[theorem]{Proposition}
\newtheorem{lemma}[theorem]{Lemma}
\newtheorem{conj}[theorem]{Conjecture}

\theoremstyle{definition}
\newtheorem{defn}[theorem]{Definition}
\newtheorem{rem}[theorem]{Remark}

\author{Liam Mazurowski}
\address{Department of Mathematics, Lehigh University, Bethlehem, Pennsylvania, 18015, United States}
\email{lim624@lehigh.edu}

\author{Xuan Yao}
\address{Department of Mathematics, Princeton University, Princeton, NJ 08540}
\email{xy1216@princeton.edu}

\title{Rigidity in the Positive Mass Theorem with $C^0$ Decay}

\begin{document}

\begin{abstract}
Let $g$ be a smooth metric on $\R^3$ with non-negative scalar curvature.    We show that if $g$ satisfies $\vert g(x)-g_{\text{euc}}(x)\vert = O(\vert x\vert^{-1-\tau})$ for some $\tau > 0$ then $g$ must be flat. 
\end{abstract}

\maketitle 

\section{Introduction}

Although scalar curvature is defined using two derivatives of the metric, there is strong evidence to suggest that non-negativity of the scalar curvature is essentially a $C^0$ property.  For example, Gromov \cite{gromov2014dirac,gromov2019four} showed that non-negative scalar curvature is preserved under $C^0$ convergence of smooth manifolds to a smooth limit.  Gromov's proof relies on the non-existence of mean convex cubes with acute dihedral angles in manifolds with non-negative scalar curvature.  In fact, angles and mean convexity can be defined for $C^0$ metrics, and the associated prism rigidity theorem \cite{brendle2024scalar,li2020polyhedron} may be taken as a possible definition of non-negative scalar curvature for $C^0$ metrics.  Bamler \cite{bamler2016ricci} gave another proof of Gromov's $C^0$-convergence theorem using Ricci flow, and  Burkhardt-Guim \cite{burkhardt2019pointwise} extended these ideas to define a synthetic notion of  non-negative scalar curvature for $C^0$ metrics based on Ricci flow. Also see the  work of Lee \cite{lee2026quantification} building on Bamler's approach. We  refer to \cite{mazurowski2026quantification} and \cite{mazurowski2026scalar} for recent extensions of Gromov's $C^0$-convergence theorem proven using harmonic functions and $\mu$-bubbles, respectively. 

The positive mass theorem \cite{schoen1979proof,schoen1981proof,witten1981new} asserts that an asymptotically flat manifold $(M,g)$ with non-negative scalar curvature must have non-negative ADM mass. Moreover, there is an associated rigidity statement: if the mass is zero then $M$ is flat.  The asymptotically flat condition in this theorem requires that $g$ converges suitably quickly in $C^2$ to the Euclidean metric at infinity, and indeed this type of decay is required to even define the classical ADM mass \cite{bartnik1986mass}.  In light of the above, it is natural to wonder whether $C^0$ convergence alone is enough to retain certain aspects of the positive mass theorem.  In particular, in Four Lectures on Scalar Curvature \cite[Section 3.11]{gromov2019four}, Gromov made the following conjecture on the rigidity of the positive mass theorem under $C^0$ convergence of the metric at infinity.

\begin{conj}[Euclidean $C^0$-Rigidity Conjecture (a)] 
\label{main-conj} Assume that $g$ is a smooth metric on $\R^3$ with non-negative scalar curvature. If $g$ satisfies $\vert g(x)-g_{\operatorname{euc}}(x)\vert = o(\vert x\vert^{-1})$ then $g$ is flat. 
\end{conj}

\begin{rem}
    Gromov also remarked that ``no available technique is capable to prove this even for very fast decay.'' 
\end{rem}

In this paper, we show that Conjecture \ref{main-conj} is true for metrics which decay slightly faster than $o(\vert x\vert^{-1})$.

\begin{theorem}
\label{theorem-main} 
Assume that $g$ is a smooth metric on $\R^3$ with non-negative scalar curvature. If $g$ satisfies $\vert g(x)-g_{\operatorname{euc}}(x)\vert = O(\vert x\vert^{-1-\tau})$ for some $\tau > 0$ then $g$ is flat. 
\end{theorem}

\begin{rem}\label{rem: decay}
It is not hard to see that for any $c > 0$ there is a smooth, non-flat metric $g$ on $\R^3$ with non-negative scalar curvature which satisfies 
\[
g = \left(1 + c \vert x\vert^{-1} + O(\vert x\vert^{-2})\right)g_{\text{euc}}.
\]
For example, one can construct such a $g$ by taking a small perturbation of the round metric on $S^3$ and then performing a conformal change by a Green's function for the conformal Laplacian. 
Thus $o(\vert x\vert^{-1})$ decay is the natural cutoff for rigidity. 
\end{rem}

\begin{rem}
We would like to note that Conjecture \ref{main-conj}  can also be deduced as a  corollary of the very nice isoperimetric mass rigidity theorem of Benatti-Fogagnolo-Mazzieri  \cite[Theorem 2.14]{benatti2025isoperimetric}. See \cite{you2026gromov} for another proof of Conjecture \ref{main-conj}  by extending the results in this work. 
\end{rem}

\begin{rem}
    The results of this paper have now been superseded by the authors' more recent work arXiv:2606.19123. 
\end{rem}

\subsection{Discussion} There is great interest in extending the positive mass theorem below the classical $C^2$ regularity threshold.  One approach studies Sobolev metrics with non-negative scalar curvature in the sense of distributions. For example, Lee and LeFloch \cite{lee2015positive} proved a positive mass theorem with rigidity for $C^0 \cap W^{1,n}$ metrics using techniques from spin geometry. 

Another approach is based on Huisken's isoperimetric mass \cite{huisken2006isoperimetric}. The isoperimetric mass is appealing because it is easy to define on manifolds with only $C^0$ regularity.  However, it was observed by Jauregui, Lee, and Unger \cite{jauregui2024note} that the isoperimetric mass is always non-negative for elementary reasons, and that this does not require non-negativity of the scalar curvature.  Nevertheless, there is still an interesting rigidity question: if a manifold has zero isoperimetric mass does this imply that it is flat? We also mention the related isocapacitary mass introduced by Jauregui \cite{jauregui2024adm}, and note that it also has applications to the low regularity setting \cite{benatti2023nonlinear}. 

Finally, there is a Ricci flow approach to studying scalar curvature for low regularity metrics. Burkhardt-Guim \cite{burkhardt2019pointwise} defined a synthetic notion of scalar curvature lower bounds for $C^0$ metrics using the Ricci flow.  Then in \cite{burkhardt2024adm}, Burkhardt-Guim defined a mass for $C^0$ asymptotically flat manifolds with non-negative scalar curvature in the Ricci flow sense, and showed that this mass is a well-defined geometric quantity, independent of any choices used in the definition.  It is still an open problem whether this mass must be non-negative \cite[Question 2]{burkhardt2024adm}. 

In this paper, we consider smooth metrics $g$ on $\R^3$ with non-negative scalar curvature in the classical sense, but we impose only $C^0$ decay assumptions at infinity.  Heuristically, because the condition 
\begin{equation}
\label{decay1}
\vert g-g_{\text{euc}}\vert = O(\vert x\vert^{-1-\tau})
\end{equation}
rules out the examples in Remark \ref{rem: decay}, one expects that any metric satisfying \eqref{decay1} will have mass zero for any reasonable low regularity generalization of the ADM mass.  If there is a $C^0$ positive mass theorem, then $g$ should be flat.  Theorem \ref{theorem-main} confirms that this is indeed the case. 

\subsection{Sketch of Proof} 
The proof of Theorem \ref{theorem-main} is based on the harmonic function method for studying scalar curvature under $C^0$ convergence developed by the authors in \cite{mazurowski2026quantification}.  Assume that $g$ is a metric on $\R^3$ with non-negative scalar curvature which satisfies the $C^0$ decay estimate 
\[
\vert g-g_{\text{euc}}\vert = O(\vert x\vert^{-1-\tau})
\]
for some $\tau > 0$.  It is well-known that such a metric $g$ admits a Green's function $u$ for $\lap_g$ with a pole at the origin.  

The first step is to obtain an asymptotic expansion for $u$ near infinity.  Let $a_{ij} = \sqrt{\vert g\vert}g^{ij}$ and set $A = (a_{ij})$.  By assumption, one has $A = I - B$ where $B = O(\vert x\vert^{-1-\tau})$.  We rewrite the equation $\lap_g u = 0$ as $\lap u = \div(B\grad u)$.  Let $X$ be a smooth vector field on $\R^3$ that agrees with $B\grad u$ outside a compact set.  Crucially, the decay rate for $B$ implies that $X\in L^1$.  Hence we can obtain a function $w$ which solves $\lap w = \div(X)$ in the sense of distributions by convolving the fundamental solution of the Laplacian with $\div(X)$.   The difference $h = u-w$ is then harmonic outside a compact set and satisfies $h\to 0$ at infinity. 

There is a well-known asymptotic expansion for any such $h$.  Moreover, we can use the explicit formula for $w$ to obtain an asymptotic expansion for $w$. Initially, this expansion for $w$ is only valid in a scale invariant $L^q$ norm.  However, returning to the PDE and using the fact that $B$ is small, we can use elliptic regularity to upgrade this to a scale invariant $W^{1,p}$ norm.  We thus obtain an expansion for $u$ of the form 
\[
u(x) = \frac{c}{\vert x\vert} + \frac{\la b + \overline X,x\ra}{\vert x\vert^3} + \hdots
\]
valid in a scale invariant $W^{1,p}$ sense. Here $c > 0$ is a constant, and $b$ and $\overline X$ are both fixed vectors in $\R^3$. 

The second step is to study various quantities associated to $u$.  The starting point is the $F$ function   introduced in \cite{agostiniani2024green}.  For each $t > 0$, the number $F(t)$ is defined as an integral of various geometric quantities over the level set $\{u=\frac 1 t\}$.  Since $g$ has non-negative scalar curvature, the $F$ function is non-decreasing and non-negative \cite{agostiniani2024green}.  However, we note that the $F$ function is highly unstable with respect to $C^0$ changes in the metric.

Thus, as in \cite{mazurowski2026quantification}, we integrate $F(t)$ twice to obtain a quantity $D(a)$ that captures information about the harmonic function $u$ and its gradient on an annulus with inner radius $a$ and outer radius $4a$.  This quantity $D(a)$ depends continuously on $u$ in $W^{1,p}$ for $p > 3$ and is therefore stable under small $C^0$ changes to the metric. Next, we prove that $D$ inherits monotonicity and non-negativity properties from $F$. More precisely, one has $D(a) \ge 0$ and the map $a\mapsto aD(a)$ is non-decreasing.  Finally, we use the above $W^{1,p}$-asymptotic expansion for $u$ to show that $aD(a) \to 0$ as $a\to \infty$.  In turn, this implies that $D(a)\equiv 0$ and hence that $F(t) \equiv 0$.  But it is known that if $F$ vanishes identically, then $g$ must be flat, and this concludes the proof. 

\begin{rem}
When $g$ is asymptotically flat in the usual sense, one can show that $aD(a)$ converges to a constant multiple of the ADM mass as $a\to \infty$.  Thus one may view $\lim_{a\to\infty} aD(a)$ as another potential generalization of the ADM mass in the low regularity setting. 
\end{rem}

\subsection{Organization} The remainder of the paper is organized as follows.  In Section \ref{section-green}, we derive the $W^{1,p}$ asymptotic expansion for the Green's function.  Then in Section \ref{section-mass}, we define the $D$ function and prove that $aD(a) \to 0$ as $a\to \infty$.  As explained above, this implies Theorem \ref{theorem-main}. Finally, in Appendix \ref{F-Appendix}, we summarize some properties of the $F$ function needed in the proof.

\subsection{Acknowledgments} 
We would like to thank Xin Zhou for his interest in this work.  We are grateful to Mattia Fogagnolo for bringing several helpful references to our attention after the first version of this work appeared.

\section{Green's Function Asymptotics} 

\label{section-green}

Assume $g$ is a smooth Riemannian metric on $\R^3$ which satisfies the $C^0$ decay estimate 
\[
\vert g(x) - g_{\text{euc}}(x)\vert = O(\vert x\vert^{-1-\tau})
\]
for some $\tau > 0$.  
Define $a_{ij} = g^{ij} \sqrt{\vert g\vert}$ and let $A = (a_{ij})$ so that $\lap_g w = 0$ if and only if $\mathcal Lw := \div(A\grad w) = 0$.  Note that $A$ is bounded and uniformly elliptic. Let $u$ be a Green's function for $\mathcal L$ with a pole at the origin.  Classical results (see for example \cite[Theorem (1.1)]{gruter1982green}) imply that $u$ exists and satisfies two sided bounds 
\begin{equation}
\label{initial-two-side-bound}
\frac{c}{\vert x\vert} \le u(x) \le \frac{C}{\vert x\vert}
\end{equation} 
for some positive constants $0 < c < C$. Multiplying $u$ by a constant if necessary, we can suppose that the flux of $\grad u$ over any level set of $u$ is $4\pi$. 
We would like to show that $u$ actually has a relatively nice asymptotic expansion near infinity.

Write $A = I - B$ where $B = O(\vert x\vert^{-1-\tau})$. We can rewrite the equation $\mathcal L u = 0$ (away from the pole) as 
\[
0 = \div(A\grad u) = \div((I-B)\grad u) = \lap u - \div(B\grad u). 
\]
Thus $u$ solves 
\[
\lap u = \div(B\grad u)
\]
away from the origin.  Let $X$ be a smooth vector field on $\R^3$ which is equal to $B\grad u$ outside of $B(0,1/16)$.  The first step is to establish some estimates on the vector field $X$.

\subsection{Estimating the Right Hand Side} 
It is convenient to introduce a family of annuli.  For $R > 0$, define $\an(R)$ to be the annulus with inner radius $R/8$ and outer radius $8R$. Clearly we have a bound 
\[
\vert u(x)\vert \le \frac{C}{R}
\]
on $\an(R)$ from \eqref{initial-two-side-bound}. This implies an $L^2$ bound on $\grad u$ using a Caccioppoli type inequality.

\begin{prop} 
\label{caccioppoli} 
We have 
\[
\int_{\an(R)} \vert \grad u\vert^2\, dx \le \frac{C}{R}
\]
where $C$ does not depend on $R$. 
\end{prop} 

\begin{proof} 
Let $\eta$ be a smooth cut-off function which is 1 on $\an(R)$, 0 inside $B(0,R/16)$, and 0 outside $B(0,16R)$.  We can arrange that $\vert \grad \eta \vert \le \frac{C}{R}$.  We test the equation for $u$ against the function $\eta^2 u$ to get 
\begin{align*}
0 = \int \la A\grad u, \grad(\eta^2 u)\ra \, dx = \int \eta^2 \la A\grad u,\grad u\ra \, dx + \int 2\eta u \la A\grad u , \grad \eta\ra \, dx.
\end{align*}
Using the fact that $A$ is uniformly elliptic and bounded, this implies that 
\[
 \int \eta^2 \vert \grad u\vert^2 \, dx \le  C  \int \eta \vert u\vert \vert \grad u\vert \vert \grad \eta\vert\, dx. 
\]
Now we use Young's inequality with $\eps$ to get 
\[
\int \eta^2 \vert \grad u\vert^2 \, dx \le \frac 1 2 \int \eta^2 \vert \grad u\vert^2\, dv + C \int u^2 \vert \grad \eta\vert^2\, dx. 
\]
This implies that 
\[
\int_{\an(R)} \vert \grad u\vert^2\, dx \le C \int u^2 \vert \grad \eta\vert^2\, dx. 
\]
Finally, we use the bounds $\vert u\vert \le \frac{C}{R}$ and $\vert \grad \eta\vert \le \frac{C}{R}$ to get 
\[
\int_{\an(R)} \vert \grad u\vert^2\, dx \le \frac{C}{R},
\]
as needed.
\end{proof} 

This implies the following integrability for $X$. 

\begin{prop}
The vector field $X$ belongs to $L^1$. 
\end{prop} 

\begin{proof}
Again consider an annulus $\an(R)$. For large $R$, we have $X = B\grad u$ and we know that $B = O(\vert x\vert^{-1-\tau})$. Therefore, for large $R$, we have 
\begin{align*}
\int_{\an(R)} \vert X\vert\, dx &\le \int_{\an(R)} \vert B\vert \vert \grad u\vert\, dx \\
&\le \frac{C}{R^{1+\tau}} \int_{\an(R)} \vert \grad u\vert\, dx \\
&\le \frac{C}{R^{1+\tau}} \left(\int_{\an(R)} \vert \grad u\vert^2\, dx\right)^{1/2} \vert \an(R)\vert^{1/2} \\
&\le \left(\frac{C}{R^{1+\tau}} \right)\left(\frac{C}{R^{1/2}}\right) \left(C R^{3/2}\right) = CR^{-\tau}.  
\end{align*}
Here we used Proposition \ref{caccioppoli} to get from the third to the fourth line. 
Next, observe that 
\[
\int_{\{\vert x\vert\ge 2\}} \vert X\vert\, dx = \sum_{k=1}^\infty \int_{\an(16^k)} \vert X\vert\, dx.
\]
So by the above bound, to get $X\in L^1$ it suffices to note that the series 
\[
\sum_{k=1}^\infty (16^k)^{-\tau} = \sum_{k=1} 16^{-\tau k}
\]
converges. 
\end{proof} 

\subsection{Convolution with the Fundamental Solution} 
Next, we find a function $w$ which satisfies $\lap w = \div(X)$ in the sense of distributions. We can obtain such a $w$ by convolving the fundamental solution of  Laplacian with $\div(X)$ and integrating by parts to take the derivatives off of $X$: 
\begin{align*}
w(x) &= -\frac 1 {4\pi}\int_{\R^3} \frac{\div X(y)}{\vert x-y\vert}\, dy \\
&= \frac 1 {4\pi} \int_{\R^3} \left\la \grad_y \frac{1}{\vert x-y\vert}, X(y)\right\ra \, dy\\
&= \frac 1 {4\pi}  \int_{\R^3} \left\la \frac{x-y}{\vert x-y\vert^3}, X(y)\right\ra \, dy.
\end{align*}
Observe that one has 
\begin{align*}
\vert w(x)\vert \le \frac 1 {4\pi} \int_{\R^3} \frac{\vert X(y)\vert}{\vert x-y\vert^2}\, dy.
\end{align*}
It is easy to see that the right hand side converges for almost every $x$ and defines an $L^1_{loc}$ function. 

\begin{prop} 
The function $w$ defined by 
\[
w(x) = \frac{1}{4\pi} \int_{\R^3}  \left\la \frac{x-y}{\vert x-y\vert^3}, X(y)\right\ra \, dy
\] 
is well-defined and belongs to $L^1_{loc}$.  Moreover, $w$ solves $\lap w = \div(X)$ in the sense of distributions. 
\end{prop} 

\begin{proof} 
Consider any ball $B(q,r)\subset \R^3$. Then by Tonelli's theorem we have 
\begin{align*}
4\pi \int_{B(q,r)} \vert w(x)\vert\, dx \le \int_{x\in B(q,r)} \int_{y\in \R^3} \frac{\vert X(y)\vert}{\vert x-y\vert^2}\, dy\, dx &= \int_{y\in \R^3} \vert X(y)\vert \int_{x\in B(q,r)} \frac{1}{\vert x-y\vert^2}\, dx\, dy. 
\end{align*}
Now for any fixed $y\in \R^3$, we have 
\begin{align*}
\int_{x\in B(q,r)} \frac{1}{\vert x-y\vert^2}\, dx &= \int_0^\infty \frac{1}{t^2} \mathcal H^2(B(q,r)\cap \{\vert x-y\vert = t\})\, dt \\
&\le 4\pi \vert \{t: B(q,r) \cap \{\vert x-y\vert = t\} \neq \emptyset\} \vert \le 8\pi r. \phantom{\int}
\end{align*} 
Thus we obtain 
\[
\int_{y\in \R^3} \vert X(y)\vert \int_{x\in B(q,r)} \frac{1}{\vert x-y\vert^2}\, dx\, dy \le 8\pi r \|X\|_{L^1}. 
\]
This shows that $w$ is a well-defined function almost everywhere and that $w\in L^1_{loc}$. It is well-known that this convolution procedure produces a distributional solution to $\lap w = \div(X)$. 
\end{proof}

Now since $X$ is in $L^1$ there is a well-defined vector 
\[
\overline X = \frac 1 {4\pi} \int_{\R^3} X(x) \, dx. 
\]
We claim that $\overline X$ governs the asymptotic behavior of $w$. In fact, $w$ should behave like the harmonic function 
$
{\la x, \overline X\ra}\vert x\vert^{-3}
$
near infinity in a suitable norm. We now make this precise. 

\begin{prop}
\label{annulus-estimate}
We have an estimate 
\[
 \frac{1}{\vert \an(R)\vert} \int_{\an(R)} \left\vert R^2 \left(w(x) - \frac{\la x, \overline X\ra}{\vert x\vert^3}\right)\right\vert^q \, dx \to 0
\]
as $R\to \infty$ for every fixed $1\le q < \frac 3 2$. 
\end{prop}

\begin{proof}
Let $\eps > 0$ be given.  Since $X$ is $L^1$, we can select some $M$ large enough that 
\[
\int_{\vert x\vert \ge M} \vert X(x)\vert\, dx < \eps^{1/q}. 
\]
By definition of $w$ and $\overline X$, we have 
\begin{align*}
w(x) - \frac{\la x, \overline X\ra}{\vert x\vert^3} &= \int_{\R^3}  \left\la \frac{x-y}{\vert x-y\vert^3}-\frac{x}{\vert x\vert^3}, \frac{X(y)}{4\pi} \right\ra \, dy\\
&=  \int_{\vert y\vert < M}  \left\la \frac{x-y}{\vert x-y\vert^3}-\frac{x}{\vert x\vert^3}, \frac{X(y)}{4\pi} \right\ra \, dy + \int_{\vert y\vert \ge M}  \left\la \frac{x-y}{\vert x-y\vert^3}-\frac{x}{\vert x\vert^3}, \frac{X(y)}{4\pi}\right\ra \, dy\\
&:= I_{1}(x) + I_{2}(x). \phantom{\int}
\end{align*}
Now suppose that $R \ge 32M$.  We have 
\begin{align*}
\frac{1}{\vert \an(R)\vert} \int_{\an(R)} \left\vert R^2 \left(w(x) - \frac{\la x, \overline X\ra}{\vert x\vert^3}\right)\right\vert^q \, dx &= \frac{1}{\vert \an(R)\vert}\int_{\an(R)} R^{2q}\vert I_{1} + I_2\vert^q\, dx \\
&\le \frac{2^q R^{2q}}{\vert \an(R)\vert} \int_{\an(R)} \max\{\vert I_1(x)\vert,\vert I_2(x)\vert\}^q \, dx \\
&\le \frac{2^{q}R^{2q}}{\vert \an(R)\vert} \int_{\an(R)}\vert I_1(x)\vert^q + \vert I_2(x)\vert^q\, dx. 
\end{align*} 
Thus it suffices to show that 
\begin{gather*}
 \frac{R^{2q}}{\vert \an(R)\vert} \int_{\an(R)}\vert I_1(x)\vert^q\, dx < C \eps,\\
  \frac{R^{2q}}{\vert \an(R)\vert} \int_{\an(R)}\vert I_2(x)\vert^q \, dx < C \eps, 
\end{gather*} 
for sufficiently large $R$.  

First consider the term with $I_1$. Suppose $x\in \an(R)$ and $\vert y\vert \le M$. 
Then observe that 
\begin{align*}
\left\vert \frac{x-y}{\vert x-y\vert^3} - \frac{x}{\vert x\vert^3} \right\vert&= \left\vert \int_0^1 \frac{d}{dt} \left(\frac{x-ty}{\vert x-ty\vert^3}\right) \, dt\right\vert\\
&\le  \int_0^1 \left\vert 3 \frac{\la x-ty,y\ra}{\vert x-ty\vert^5}(x-ty) - \frac{y}{\vert x-ty\vert^3}\right\vert\, dt \le \frac{CM}{R^3}. 
\end{align*}
Hence we obtain 
\[
\left\vert \int_{\vert y\vert < M} \left\la \frac{x-y}{\vert x-y\vert^3}-\frac{x}{\vert x\vert^3}, \frac{X(y)}{4\pi}\right\ra \, dy\right\vert \le \frac{C M}{R^3} \|X\|_{L^1}. 
\]
Thus we have 
\begin{align*}
\frac{R^{2q}}{\vert \an(R)\vert} \int_{\an(R)} \vert I_1(x)\vert^q\, dx \le \frac{C R^{2q} M^{q} \|X\|_{L^1}^q}{R^{3q}} = \frac{CM^q \|X\|_{L^1}^q}{R^q}
\end{align*}
and this is less than $\eps$ for sufficiently large $R$. 

Now we estimate the term with $I_2$.  We note that 
\begin{align*}
\vert I_2(x)\vert &= \left\vert \int_{\vert y\vert \ge M}  \left\la \frac{x-y}{\vert x-y\vert^3}-\frac{x}{\vert x\vert^3}, \frac{X(y)}{4\pi}\right\ra \, dy\right\vert \\
&\le  \int_{\vert y\vert \ge M} \frac{\vert X(y)\vert}{4\pi \vert x\vert^2}\, dy +  \int_{\vert y\vert \ge M}  \frac{\vert X(y)\vert}{4\pi \vert x-y\vert^2} \, dy \\
&:= I_3(x) + I_4(x). \phantom{\int}   \end{align*}
Therefore by the same reasoning as above we have 
\begin{align*}
\frac{R^{2q}}{\vert \an(R)\vert} \int_{\an(R)} \vert I_2(x)\vert^q \, dx \le \frac{2^q R^{2q}}{\vert \an(R)\vert} \int_{\an(R)} \vert I_3(x)\vert^q + \vert I_4(x)\vert^q\, dx.  
\end{align*} 
Next, observe that 
\begin{align*}
\vert I_3(x)\vert =  \frac{1}{4\pi} \int_{\vert y\vert\ge M} \frac{\vert X(y)\vert}{\vert x\vert^2} \,dy \le \frac{C}{R^2} \int_{\vert y\vert \ge M} \vert X(y)\vert \, dy \le \frac{C}{R^2} \eps^{1/q}. 
\end{align*}
It follows that 
\[
\frac{R^{2q}}{\vert \an(R)\vert} \int_{\an(R)} \vert I_3(x)\vert^q\, dx \le \frac{C R^{2q} \eps}{R^{2q}} = C\eps. 
\]
Finally, it remains to consider the term with $I_4$.  Observe that 
\begin{align*}
\left[\frac{R^{2q}}{\vert \an(R)\vert} \int_{\an(R)} \vert I_4(x)\vert^q\, dx\right]^{1/q} &= \left[\int_{\an(R)} \left[\int_{\vert y\vert \ge M} \frac{R^2 \vert X(y)\vert}{4\pi \vert x-y\vert^2}\, dy \right]^q\, \frac{dx}{\vert \an(R)\vert} \right]^{1/q}\\
&\le \int_{\vert y\vert \ge M} \left[\int_{\an(R)} \left(\frac{R^2}{4\pi}\right)^q \frac{\vert X(y)\vert^q}{\vert x-y\vert^{2q}}\, \frac{dx}{\vert \an(R)\vert}\right]^{1/q}\, dy
\end{align*} 
by the Minkowski inequality for integrals.  
Now for any fixed $y$, we compute  
\begin{align*} 
\int_{\an(R)} \frac{1}{\vert x-y\vert^{2q}}\, dx &= \int_{\an(R)\cap B(y,R)} \frac{1}{\vert x-y\vert^{2q}}\, dx + \int_{\an(R) - B(y,R)} \frac{1}{\vert x-y\vert^{2q}}\, dx \\
&\le \int_{B(y,R)} \frac{1}{\vert x-y\vert^{2q}}\, dx + R^{-2q} \vert \an(R)\vert \\
&\le C R^{3-2q} \phantom{\int} 
\end{align*} 
where $C$ does not depend on $y$ and we used $1\le q < \frac 3 2$.  Hence we obtain 
\begin{align*}
&\int_{\vert y\vert \ge M} \left[\int_{\an(R)} \left(\frac{R^2}{4\pi}\right)^q \frac{\vert X(y)\vert^q}{\vert x-y\vert^{2q}}\, \frac{dx}{\vert \an(R)\vert}\right]^{1/q}\, dy \\
&\qquad = \frac{R^2}{4\pi}\int_{\vert y\vert \ge M}  \vert X(y)\vert \left[ \frac{1}{\vert \an(R)\vert} \int_{\an(R)} \frac{1}{\vert x-y\vert^{2q}}\, dx\right]^{1/q}\, dy \\
&\qquad \le C R^2 \int_{\vert y\vert\ge M} \vert X(y)\vert (R^{-2q})^{1/q}\, dy \\
&\qquad = C \int_{\vert y\vert\ge M} \vert X(y)\vert\, dy \le C\eps^{1/q}. 
\end{align*}
It follows that 
\[
\frac{R^{2q}}{\vert \an(R)\vert} \int_{\an(R)} \vert I_4(x)\vert^q\, dx \le C\eps
\]
for large $R$.  
Putting everything together, we see that 
\[
\frac{1}{\vert \an(R)\vert} \int_{\an(R)} \left\vert R^2 \left(w(x) - \frac{\la x, \overline X\ra}{\vert x\vert^3}\right)\right\vert^q\, dx \le  C\eps
\]
for large $R$.  This proves the result. 
\end{proof} 

\subsection{The Harmonic Remainder} 
Next, define $h = u - w$.  Outside of a compact set containing the origin, the function $h$ solves $\lap h = 0$ in the sense of distributions, and therefore $h$ is actually a smooth, harmonic function.  We now show that $h$ decays suitably at infinity. 

\begin{prop} 
We have 
\[
\frac{1}{\vert \an(R)\vert} \int_{\an(R)} \vert h\vert\, dx \to 0
\]
as $R\to \infty$. 
\end{prop} 

\begin{proof} 
Observe that 
\begin{align*}
\frac{1}{\vert \an(R)\vert}\int_{\an(R)}  \vert h\vert\, dx &\le \frac{1}{\vert \an(R)\vert} \int_{\an(R)} \vert u\vert + \vert w\vert\, dx \\
&\le \frac{1}{\vert \an(R)\vert} \int_{\an(R)} \frac{C}{R} + \left\vert w(x) - \frac{\la x,\overline X\ra}{\vert x\vert^3}\right\vert+ \left\vert \frac{\la x,\overline X\ra}{\vert x\vert^3}\right\vert \, dx\\
&\le \frac{1}{\vert \an(R)\vert} \int_{\an(R)} \frac{C}{R} + \left\vert w(x) - \frac{\la x,\overline X\ra}{\vert x\vert^3}\right\vert  + \frac{C}{R^2}\, dx.
\end{align*} 
This goes to 0 as $R\to \infty$ by Proposition \ref{annulus-estimate} with $q=1$. 
\end{proof}

 In fact, this implies that $h$ goes to 0 at infinity in a pointwise sense. 

\begin{prop}
We have $\vert h(x)\vert \to 0$ as $\vert x\vert \to \infty$. 
\end{prop}

\begin{proof}
Select a point $x$ with $\vert x\vert = 2R$ and note that $B(x,R) \subset \an(R)$. By the mean value property of harmonic functions, we have 
\[
\vert h(x)\vert \le \frac{1}{\vert B(x,R)\vert}\int_{B(x,R)} \vert h(y)\vert\, dy \le \frac{C}{\vert \an(R)\vert} \int_{\an(R)} \vert h(y)\vert\, dy. 
\]
We have seen that the average on the right goes to 0, and it follows that $\vert h(x)\vert\to 0$ as $\vert x\vert\to \infty$. 
\end{proof}

Now since $h$ is a harmonic function outside a compact set which goes to 0 at infinity, there is an expansion 
\[
h(x) = \frac{c}{\vert x\vert} + \frac{\la b,x\ra}{\vert x\vert^3} + O(\vert x\vert^{-3}), \quad \text{as } \vert x\vert\to \infty
\]
for some $c > 0$ and some $b\in \R^3$.

\subsection{The Asymptotic Expansion}

Next, we want to obtain convergence in a stronger space. Recall that $w = u - h$. Define 
\[
v = w - \frac{\la x, \overline X\ra}{\vert x\vert^3} = u - h - \frac{\la x, \overline X\ra}{\vert x\vert^3}.
\]
Note that $\la x,\overline X\ra \vert x\vert^{-3}$ is harmonic. 
Therefore we have 
\begin{align*}
\div(A\grad v) &= \div(A\grad u) - \div(A\grad h) - \div\left(A \grad \frac{\la x,\overline X\ra}{\vert x\vert^3}\right)  \\
&= - \div(A\grad h) - \div\left(A \grad \frac{\la x,\overline X\ra}{\vert x\vert^3}\right)\\
& = \div\left(B\grad \left( h +  \frac{\la x,\overline X\ra}{\vert x\vert^3}\right) \right). 
\end{align*}
 Define the vector field 
 \[
 Y = B\grad \left(h + \frac{\la x,\overline X\ra}{\vert x\vert^3}\right).
 \]
Using $B = O(\vert x\vert^{-1-\tau})$ and the asymptotic expansion for $h$, we see that 
\[
\vert Y\vert = O(\vert x\vert^{-3-\tau}). 
\]
It is now convenient to rescale everything to a fixed annulus.   

Define  $w_R(y) = R^2 w(Ry)$ and $v_R(y) = R^2 v(Ry)$.   Using the change of variables $y = x/R$, the scale invariant annular estimate for $w$ (Proposition \ref{annulus-estimate}) becomes 
\[
\int_{\an(1)} \left\vert w_R(y) - \frac{\la y, \overline X\ra}{\vert y\vert^3}\right\vert^q\, dy = \int_{\an(1)} \vert v_R(y)\vert^q\, dy \to 0. 
\]
Further define $A_R(y) = A(Ry)$ and $Y_R(y) = R^3Y(Ry)$.
Then the equation $\div(A\grad v) = \div(Y)$ becomes 
\begin{equation}
\label{v-eqn}
\div\left(A_R(y)\grad  v_R(y)\right) = \div(Y_R(y))
\end{equation}
after rescaling. 
Moreover, we have the decay estimates $\vert A_R(y) - I\vert \le CR^{-1-\tau}$ and $\vert Y_R(y)\vert \le C R^{-\tau}$ as $R\to \infty$. 

Define $\an_2 = B(0,4)-B(0,1)$ and $\an_1 = B(0,5)-B(0,1/4)$ so that $\an_2 \subset\subset \an_1 \subset\subset \an(1)$. Fix some $p > 3$.  Since $v$ satisfies \eqref{v-eqn} and $A_R$ is arbitrarily close to $I$ for large $R$, Meyers' estimate \cite[Theorem 2]{meyers1963p} shows that 
\[
\|\grad v_R\|_{L^p(\an_2)} \le C (\|v_R\|_{L^2(\an_1)} + \|Y_R\|_{L^\infty(\an_1)}).
\]
Moreover, the De Giorgi-Nash-Moser theory (see \cite[Theorem 8.17]{gilbarg1998elliptic}) implies that 
\[
\|v_R\|_{L^\infty(\an_1)} \le C(\|v_R\|_{L^q(\an(1))} + \|Y_R\|_{L^\infty(\an(1))})
\]
for any fixed $q > 1$. 
Combining this with the previous estimate and 
\begin{gather*} 
\|v_R\|_{L^q(\an(1))}\to 0 \text{ for } 1 \le q < \frac 3 2,\\ \|Y_R\|_{L^\infty(\an(1))} \to 0,
\end{gather*}
 we obtain $v_R \to 0$ in $W^{1,p}(\an_2)$.

Let us summarize the consequences for the function $u$.  From now on, let us write $\an = \an_2 = B_4 - B_1$.  This is the only annulus we need to consider in the remainder of the paper. Recall that $u$ is equal to $w + h$.  Define the rescaled function $u_R(y) = Ru(Ry)$.  Then for $y\in \an$ we have 
\begin{align*}
u_R(y) &= Rw(Ry) + Rh(Ry) \\
&= \frac{c}{\vert y\vert} + \frac{1}{R}\left[\frac{\la b,y\ra}{\vert y\vert^3} + w_R(y)\right] + O(R^{-2}). 
\end{align*}
Moreover, the $O(R^{-2})$ term comes purely from the asymptotic expansion of $h$ and so we also have good estimates for the derivative of this term.  

Recall that we are assuming that the flux of $\grad u$ over any level set is $4\pi$. This implies that $c = 1$.   

\begin{prop}
    In the above formula, one has $c = 1$. 
\end{prop}

\begin{proof} Define metrics $g_R(y) = g(Ry)$ and note that $g_R \to g_{\text{euc}}$ in $C^0$ on $\an$ as $R\to \infty$.
For almost-every $t\in (\frac{c}{3},\frac{c}{2})$, we have 
\[
\int_{\{u_R=t\}} \vert \grad^{g_R} u_R(y)\vert \, da_{g_R}(y)= \int_{\{u=t/R\}} \vert \grad^g u(x)\vert \, da_g(x) = 4\pi
\]
since we assumed that the flux of $\grad u$ over any level set is $4\pi$. Let $\eta$ be a smooth, non-negative, bump function supported in $(\frac c 3,\frac c 2)$ with 
\[
\int_{c/3}^{c/2} \eta(t)\, dt = 1. 
\]
The $W^{1,p}$ convergence with $p > 3$ implies that $u_R \to c\vert y\vert^{-1}$ in $C^0$.  Thus the co-area formula implies that 
\begin{align*}
    \int_{B(0,3)-B(0,2)} \eta(u_R)\vert \grad u_R\vert^2\,d v_{g_R} = \int_{c/3}^{c/2} \eta(t) \left[\int_{\{u_R = t\}} \vert \grad u_R\vert \,da_{g_R}\right]\, dt = 4\pi. 
\end{align*}
On the other hand, since $u_R\to c\vert y\vert^{-1}$ in $W^{1,p}$ and $g_R\to g_{\text{euc}}$ in $C^0$, we have 
\[
\int_{B(0,3)-B(0,2)} \eta(u_R)\vert \grad u_R\vert^2\,d v_{g_R} \to \int_{B(0,3)-B(0,2)} \eta\left(\frac{c}{\vert y\vert}\right) \frac{c^2}{\vert y\vert^4}\, dy.
\]
The co-area formula gives 
\[
\int_{B(0,3)-B(0,2)} \eta\left(\frac{c}{\vert y\vert}\right) \frac{c^2}{\vert y\vert^4}\, dy = 4\pi \int_2^3 \eta\left(\frac{c}{r}\right)\frac{c^2}{r^2}\, dr = 4\pi c \int_{c/3}^{c/2} \eta(t) \,dt = 4\pi c, 
\]
and hence it must be that $c=1$. 
\end{proof}

Thus we now have demonstrated the following: 

\begin{prop} 
\label{asymptotics} 
The functions $u_R$ satisfy 
\[
u_R(y) \to \frac{1}{\vert y\vert}, \quad R\left(u_R(y) - \frac 1 {\vert y\vert}\right) \to \frac{\la b + \overline X, y\ra}{\vert y\vert^3}
\]
as $R\to \infty$ in $W^{1,p}(\an)$. 
\end{prop}

\section{Estimating the Mass} 

\label{section-mass}

Assume $g$ is a smooth Riemannian metric on $\R^3$ which satisfies the $C^0$ decay estimate 
\[
\vert g(x) - g_{\text{euc}}(x)\vert = O(\vert x\vert^{-1-\tau})
\]
for some $\tau > 0$. Assume in addition now that $g$ has non-negative scalar curvature.  Let $u$ be a Green's function for $\lap_g$ with pole at the origin as in the previous section. 

\subsection{Definitions} 
Consider the function 
\[
F(t) = 4\pi t - t^2 \int_{\{ u = \frac 1 t\}} H_g\vert \grad^g u\vert\ da_g + t^3 \int_{\{u = \frac 1 t \}} \vert \grad^g u\vert^2 \, da_g. 
\]
See Appendix \ref{F-Appendix} for a summary of the basic properties of the $F$ function, which was first introduced in \cite{agostiniani2024green}. In particular, note that since $g$ has non-negative scalar curvature, the function $F$ is monotone non-decreasing and satisfies $F \ge 0$.  

Following \cite{mazurowski2026quantification}, we now define the $E$ and $D$ functions. First, for $a > 0$ and $s\in[0,a]$, the $E$ function is given by  
\begin{align*}
E(a,s) &= \int_{a+s}^{2a+2s} \frac{F(t)}{t^3}\, dt. 
\end{align*} 
As in \cite[Section 4]{mazurowski2026quantification}, one can use the identity 
\[
\frac{d}{dt} \int_{\{u=\frac 1 t\}} \vert \grad^g u\vert^2\, da_g = -t^{-2}\int_{\{u=\frac 1 t\}} H_g \vert \grad^g u\vert\, da_g
\]
from Proposition \ref{ac-gradient} to verify that 
\[
E(a,s) = \frac{2\pi}{a+s} + \int_{\{u = \frac{1}{2a+2s}\}} \frac{\vert \grad^g u\vert^2}{u}\, da_g - \int_{\{u = \frac{1}{a+s}\}} \frac{\vert \grad^g u\vert^2}{u}\, da_g. 
\]
Next we define the $D$ function. 

In fact, it is convenient to slightly modify the definition of the $D$ function from \cite{mazurowski2026quantification} so that it is easier to work with later.  Fix a smooth, non-negative function $\psi \colon \R\to \R$ which is compactly supported in $(0,1)$ and satisfies 
\[
\int_{0}^1 \psi(s)\, ds = 1. 
\]
Then for $a > 0$ define 
\[
D(a) = \int_0^a \psi(s/a) E(a,s)\, ds. 
\]
Define the constant 
\[
c_\psi = 2\pi \int_0^1 \frac{\psi(s)}{1+s}\, ds.
\]
Using the co-area formula, we compute that 
\begin{align*}
D(a) &= c_\psi + \frac{1}{2} \int_{\{\frac 1 {4a} \le u \le \frac{1}{2a}\}} \psi\left(\frac{1}{2au} - 1\right) \frac{\vert \grad^g u\vert^3}{u^3}\, dv_g - \int_{\{\frac{1}{2a} \le u \le \frac 1 a \}} \psi\left(\frac{1}{au}-1\right) \frac{\vert \grad^g u\vert^3}{u^3}\, dv_g\\
&= c_\psi + \frac{1}{2} \int \psi\left(\frac{1}{2au} - 1\right) \frac{\vert \grad^g u\vert^3}{u^3}\, dv_g - \int \psi\left(\frac{1}{au}-1\right) \frac{\vert \grad^g u\vert^3}{u^3}\, dv_g\\
&= c_\psi + \int \left[\frac 1 2 \psi\left(\frac{1}{2au} - 1\right) - \psi\left(\frac{1}{au} - 1\right)\right] \frac{\vert \grad^g u\vert^3}{u^3}\, dv_g,
\end{align*} 
where we used the fact that $\psi$ is compactly supported in $(0,1)$ to get from the firs to the second  line. 

\subsection{Monotonicity} 
The quantity $D(a)$ inherits monotonicity and non-negativity properties from $F$.  More precisely, we have the following results. 

\begin{prop}
We have $D(a) \ge 0$ for all $a > 0$. 
\end{prop}

\begin{proof}
This follows immediately from the fact that $F(t) \ge 0$ for all $t > 0$. 
\end{proof}

\begin{prop}
The function $a \mapsto a D(a)$ is non-decreasing. 
\end{prop}

\begin{proof}
We have 
\[
a D(a) = \int_0^a a \psi(s/a)\left[ \int_{a+s}^{2a+2s} \frac{F(t)}{t^3}\,dt\right]\, ds. 
\]
In the inner integral, we make the change of variables $t = u(a+s)$ to get 
\[
\int_{a+s}^{2a+2s} \frac{F(t)}{t^3}\,dt = \frac{1}{(a+s)^2} \int_1^2 \frac{F((a+s)u)}{u^3}\, du.
\]
Thus we get 
\[
aD(a) = \int_0^a \frac{a \psi(s/a)}{(a+s)^2} \left[ \int_1^2 \frac{F((a+s)u)}{u^3}\, du\right]\, ds. 
\]
Now we make the substitution $v = s/a$ in the outer integral to get 
\[
aD(a) = \int_0^1 \frac{\psi(v)}{(1+v)^2} \left[\int_1^2 \frac{F((1+v)au)}{u^3}\, du\right]\, dv. 
\]
Therefore $aD(a)$ is a non-decreasing function because $F$ is a non-decreasing function. 
\end{proof}

\subsection{Limits at Infinity} We now want to prove Theorem \ref{theorem-main}.  In fact, it suffices to show that $aD(a) \to 0$.  Indeed, if $aD(a) \to 0$, then it follows from the monotonicity that $D(a) \equiv 0$ and therefore that $F(t) \equiv 0$.  But it is known that $F(t)\equiv 0$ implies that $g$ is flat; see Proposition \ref{F-rigidity}.   

We now turn to the proof that $aD(a) \to 0$ as $a\to \infty$. 
It is convenient to rescale everything to a fixed annulus.  Define $u_a(y) = a u(ay)$ and $g_a(y) = g(ay)$.  Then we have 
\[
D(a) = c_\psi + \int \left[\frac 1 2 \psi\left(\frac{1}{2u_a(y)} - 1\right) - \psi\left(\frac{1}{u_a(y)}-1\right)\right] \frac{\vert \grad^{g_a} u_a(y)\vert^3}{u_a(y)^3} \, dv_{g_a}(y). 
\]
For notational convenience, let 
\[
\varphi(t) = \frac{1}{t^3}\left[\frac 1 2 \psi\left(\frac 1{2t}-1\right) - \psi\left(\frac 1 t - 1\right)\right] 
\]
so that 
\[
D(a) = c_\psi + \int \varphi(u_a(y)) {\vert \grad^{g_a} u_a(y)\vert^3}\, dv_{g_a}(y). 
\]
We now study this quantity using the expansions from the previous section. 

First we make an observation.  Recall that $u_a \to \vert y\vert^{-1}$ in $W^{1,p}(\an)$ as $a\to \infty$. In particular, this implies that $u_a \to \vert y\vert^{-1}$ in $C^0(\an)$. Since $\varphi$ is compactly supported in $(1/4,1)$, it follows that 
\[
D(a) = c_\psi + \int_{\an} \varphi(u_a(y)) {\vert \grad^{g_a} u_a(y)\vert^3}\, dv_{g_a}(y)
\]
for sufficiently large $a$. In other words, we can take the domain of integration to be the fixed compact set $\an = B(0,4) - B(0,1)$ that does not depend on $a$.  

We also introduce some notation. Let $\mathcal G(\an)$ be the set of all continuous Riemannian metrics on the compact set $\an$.  We define a functional 
\begin{gather*}
\mathcal D\colon \mathcal G(\an) \times W^{1,p}(\an) \to \R,\\
\mathcal D(h,f) = c_\psi + \int_{\an} \varphi(f) {\vert \grad^{h} f\vert^3} \, dv_h. 
\end{gather*} 
It is straightforward to verify that $\mathcal D$ is jointly continuous in the pair $(h,f)$ provided $p > 3$.  This immediately implies the following. 

\begin{prop}
We have $D(a) \to 0$ as $a\to \infty$. 
\end{prop}

\begin{proof}
Note that $D(a) = \mathcal D(g_a,u_a)$ and that $g_a \to g_{\text{euc}}$ in $C^0$ and $u_a \to \vert y\vert^{-1}$ in $W^{1,p}$.  Thus we have 
\[
D(a) = \mathcal D(g_a,u_a) \to \mathcal D(g_{\text{euc}}, \vert y\vert^{-1}) = 0, \quad \text{as } a\to \infty
\]
by the continuity of $\mathcal D$. 
\end{proof}

To show that $aD(a)\to 0$, we need to study the linearization of $\mathcal D$.  We first record a simple lemma that will be useful in the proof. 

\begin{lemma}
For any vectors $X,Y\in \R^3$ one has 
\[
\big\vert \vert X+Y\vert^3 - \vert X\vert^3 - 3 \vert X\vert \la X,Y\ra\big \vert \le C (\vert X\vert \vert Y\vert^2 + \vert Y\vert^3) 
\]
where $C$ does not depend on $X$ and $Y$. 
\end{lemma}

\begin{proof}
We calculate that 
\begin{align*}
\vert X + Y\vert^3 - \vert X\vert^3 &= \int_0^1 \frac{d}{dt} \la X+tY,X+tY\ra^{3/2} \, dt \\
&= 3 \int_0^1 \la X+tY,X+tY\ra^{1/2} \la X+tY,Y\ra\, dt.
\end{align*}
Then we estimate 
\begin{align*}
\int_0^1\big\vert \vert X+tY\vert \la X+tY,Y\ra - \vert X\vert \la X,Y\ra \big\vert\, dt &\le \int_0^1 \big\vert \vert X+tY\vert - \vert X\vert\big\vert \vert X\vert \vert Y\vert + t \vert X+tY\vert \vert Y\vert^2\, dt \\
&\le C \int_0^1 \vert X\vert \vert Y\vert^2 + \vert Y\vert^3\, dt\\
&= C(\vert X\vert \vert Y\vert^2 + \vert Y\vert^3). \phantom{\int} 
\end{align*}
The result follows. 
\end{proof} 

Now we proceed to show the differentiability of $\mathcal D$.   

\begin{prop}
The map $\mathcal D$ is Frechet differentiable at the point $(\delta, \rho) = (g_{\operatorname{euc}},\vert y\vert^{-1})$ with 
\begin{align*}
\mathcal D'_{(\delta,\rho)} (k,v) = L(k,v) &:= \int_{\an} \varphi'(\rho)  \vert \grad \rho\vert^3 v\, dy \\
&\qquad + 3 \int_{\an} \varphi(\rho) \vert \grad \rho\vert \la \grad \rho, \grad v\ra \, dy\\
&\qquad + \int_{\an} \varphi(\rho) \left[\frac 1 2 (\tr k) \vert \grad \rho\vert^3 - \frac 3 2 \vert \grad \rho\vert k(\grad \rho,\grad \rho)\right]\, dy. 
\end{align*}
\end{prop}

\begin{proof}
Note that any symmetric 2-tensor sufficiently close to $\delta$ in $C^0$ will be uniformly positive definite and that any function $f$ sufficiently close to $\rho$ in $W^{1,p}$ will satisfy $f \ge c > 0$ for a constant $c$ that does not depend on $f$.  
We proceed in several steps. 

{\bf Step 1.} First, we freeze the metric at the Euclidean metric and prove the Frechet differentiability of the functional 
\[
J(f) = \mathcal D(\delta,f) = c_\psi + \int_{\an} \varphi(f) \vert \grad f\vert^3\, dy. 
\]
at $f = \rho$.  We claim that 
\[
J'_{\rho}(v) = L_1(v) := \int_{\an} \varphi'(\rho)  \vert \grad \rho\vert^3 v\, dy 
+ 3 \int_{\an} \varphi(\rho) \vert \grad \rho\vert \la \grad \rho, \grad v\ra \, dy.
\]
We need to show that 
\[
J(\rho+v) - J(\rho) - L_1(v) = o(\|v\|_{W^{1,p}}). 
\]
To this end, we observe that there is a uniform estimate 
\[
\varphi(\rho + v) = \varphi(\rho) + \varphi'(\rho)v + O(\|v\|_{C^0}^2)
\]
once $\|v\|_{C^0}$ is sufficiently small. Also, by the previous lemma, there is a uniform estimate 
\[
\vert \grad \rho + \grad v\vert^3 = \vert \grad \rho\vert^3 + 3 \vert \grad \rho\vert \la \grad \rho, \grad v\ra + O(\vert \grad \rho\vert \vert \grad v\vert^2 + \vert \grad v\vert^3). 
\]
It follows that 
\begin{align*}
J(\rho+v)-J(\rho) &= \int_{\an} \varphi'(\rho) \vert \grad \rho\vert^3 v + 3\varphi(\rho)\vert \grad \rho\vert \la \grad \rho,\grad v\ra\, dy \\
&\qquad + \int_{\an} 3\varphi'(\rho)\vert \grad \rho\vert v \la \grad \rho,\grad v\ra + (\vert \grad \rho\vert^3 + 3\vert \grad \rho\vert \la \grad \rho,\grad v\ra)O(\|v\|_{C^0}^2)\, dy \\
&\qquad + \int_{\an} (\varphi(\rho) + \varphi'(\rho)v + O(\|v\|_{C^0}^2)) O(\vert \grad \rho\vert \vert \grad v\vert^2 + \vert \grad v\vert^3)\, dy. 
\end{align*}
It is easy to see that the remainder terms in the second and third lines are in fact uniformly $O(\|v\|_{W^{1,p}}^2)$. This proves the differentiability of $J$. 

{\bf Step 2.} Next, define 
\[
L_2(k;f) =  \int_{\an} \varphi(f) \left[\frac 1 2 (\tr k) \vert \grad f\vert^3 - \frac 3 2 \vert \grad f\vert k(\grad f,\grad f)\right]\, dy. 
\]
We claim that 
\[
\mathcal D(\delta + k, \rho+v) - \mathcal D(\delta,\rho+v) - L_2(k,\rho + v) = o(\|k\|_{W^{1,p}}). 
\]
Define the function 
\[
\Psi(H,\xi) = \la H^{-1}\xi,\xi\ra^{3/2} \sqrt{\det H}
\]
where $H$ is a positive definite symmetric matrix and $\xi\in S^2$.  This is a smooth function of $(H,\xi)$ for $H$ in a neighborhood of the identity matrix and $\xi \in S^2$.  By Taylor expansion, for a symmetric matrix $K$ close to 0, we have  
\[
\Psi(I+K,\xi) = \Psi(I,\xi) + \Psi'_{(I,\xi)}(K,0) + O(\|K\|^2)
\]
where the constant in $O(\|K\|^2)$ does not depend on $\xi$. Now since $\Psi$ is homogeneous in $\xi$, we obtain 
\[
\Psi(I+K,\xi) = \Psi(I,\xi) + \Psi'_{(I,\xi)}(K,0) + \vert \xi\vert^3 O(\|K\|^2)
\]
for any $\xi \neq 0$ where again the constant in $O(\|K\|^2)$ doesn't depend on $\xi$. Finally, we note that this equation is also valid for $\xi = 0$ since then every term is equal to $0$.  

We compute that 
\[
\Psi'_{(I,\xi)}(K,0) =  \frac 1 2 \vert \xi\vert^3 \tr(K) - \frac 3 2 \vert \xi\vert \la K\xi,\xi\ra.
\]
Now observe that 
\begin{align*}
&\mathcal D(\delta+k,\rho+v)-\mathcal D(\delta,\rho+v) \\
&\qquad = \int_{\an} \varphi(\rho+v) [\Psi(\delta+k,\grad \rho + \grad v) - \Psi(\delta,\grad \rho + \grad v)]\, dy\\
&\qquad = \int_{\an} \varphi(\rho+v) \left[ \frac 1 2 \vert \grad \rho + \grad v\vert^3 \tr(k) - \frac 3 2 \vert \grad \rho + \grad v\vert  k(\grad \rho + \grad v,\grad \rho + \grad v)\right]\, dy  \\
&\qquad \qquad \qquad   + \int_{\an} \varphi(\rho+v)\left[ \vert \grad \rho + \grad v\vert^3 O(\|k\|_{C^0}^2)\right]\, dy\\
& \qquad = L_2(k;\rho+v) + C \|k\|_{C^0}^2
\end{align*}
where we used the fact that 
\[
\int_{\an} \vert \grad \rho + \grad v\vert^3 \, dy \le \int_{\an} \vert \grad \rho\vert^3 + 3 \vert \grad \rho\vert^2 \vert \grad v\vert + 3 \vert \grad \rho\vert \vert \grad v\vert^2 + \vert \grad v\vert^3\, dy \le C. 
\]
This completes Step 2.  

{\bf Step 3:} Combining Steps 1 and 2, we deduce that 
\[
\mathcal D(\delta+k,\rho+v)-\mathcal D(\delta,\rho) = L_1(v) + L_2(k;\rho+v) + o(\|v\|_{W^{1,p}} + \|k\|_{C^0}). 
\]
To complete the proof we need to show that 
\[
L_2(k;\rho+v) = L_2(k,\rho) + o(\|v\|_{W^{1,p}}+\|k\|_{C^0}).
\] 
We now verify this estimate.  
We compute that 
\begin{align*}
&\vert L_2(k,\rho+v) - L_2(k,\rho)\vert\\
&\qquad \le C \|k\|_{C^0} \int_{\an} \left\vert \varphi(\rho+v) \vert \grad \rho + \grad v\vert^3 - \varphi(\rho)\vert \grad \rho\vert^3\right\vert \,dy \\
&\qquad \quad + C \int_{\an} \left\vert \varphi(\rho+v) \vert \grad \rho + \grad v\vert  k(\grad \rho + \grad v,\grad \rho+\grad v) - \varphi(\rho)\vert \grad \rho\vert k(\grad \rho,\grad \rho)\right\vert\, dy.
\end{align*} 
Let us estimate the two integrals separately. For the first, observe that 
\begin{align*}
&\left \vert \varphi(\rho+v) \vert \grad \rho + \grad v\vert^3 - \varphi(\rho)\vert \grad \rho\vert^3\right\vert \\
&\qquad \le \vert \varphi(\rho+v) - \varphi(\rho)\vert \vert \grad \rho + \grad v\vert^3 + \vert \varphi(\rho)\vert \left\vert \vert \grad \rho+\grad v\vert^3 - \vert \grad \rho\vert^3\right\vert\\
&\qquad \le C \|v\|_{C^0}(1 + \vert \grad v\vert + \vert \grad v\vert^2 + \vert \grad v\vert^3) + C (1+\vert \grad v\vert)^2\vert \grad v\vert
\end{align*}
which yields 
\[
C \|k\|_{C^0}  \int_{\an} \left\vert \varphi(\rho+v) \vert \grad \rho + \grad v\vert^3 - \varphi(\rho)\vert \grad \rho\vert^3\right\vert \,dy = o(\|v\|_{W^{1,p}} + \|k\|_{C^0})
\]
upon integrating. Now consider the second integral.  Again we use the triangle inequality to get 
\begin{align*}
&\left\vert \varphi(\rho+v) \vert \grad \rho + \grad v\vert  k(\grad \rho + \grad v,\grad \rho+\grad v) - \varphi(\rho)\vert \grad \rho\vert k(\grad \rho,\grad \rho)\right\vert\\
&\qquad \le \vert \varphi(\rho+v)-\varphi(\rho)\vert \vert \grad \rho + \grad v\vert \vert k(\grad \rho + \grad v,\grad \rho+\grad v) \vert \\
&\qquad \qquad + \vert \varphi(\rho)\vert \vert \grad \rho + \grad v\vert - \vert \grad \rho\vert\vert \vert k(\grad \rho + \grad v,\grad \rho+\grad v)\vert \\
&\qquad \qquad + \vert \varphi(\rho)\vert \vert \grad \rho\vert \vert k(\grad \rho+\grad v,\grad \rho+\grad v)-k(\grad \rho,\grad \rho)\vert. 
\end{align*} 
The term on the second line is bounded by 
\[
C \|v\|_{C^0} \|k\|_{C^0} (1+\vert \grad v\vert)^3,
\]
and the term on the third line is bounded by 
\[
C \vert \grad v\vert \|k\|_{C^0} (1+\vert \grad v\vert)^2. 
\]
Finally, the term on the fourth line is bounded by 
\[
C \vert 2k(\grad \rho,\grad v) + k(\grad v,\grad v)\vert \le C \|k\|_{C^0} \vert \grad v\vert + C\|k\|_{C^0} \vert \grad v\vert^2. 
\]
Combining these estimates and integrating show that 
\begin{align*}
&\int_{\an} \left\vert \varphi(\rho+v) \vert \grad \rho + \grad v\vert  k(\grad \rho + \grad v,\grad \rho+\grad v) - \varphi(\rho)\vert \grad \rho\vert k(\grad \rho,\grad \rho)\right\vert\, dy\\
&\qquad  = o(\|v\|_{W^{1,p}} + \|k\|_{C^0}). 
\end{align*} 
Thus we have $L_2(k;\rho+v) = L_2(k;\rho) + o(\|v\|_{W^{1,p}} + \|k\|_{C^0})$, as desired.  Hence we have now demonstrated that $\mathcal D$ is differentiable with derivative $\mathcal D'_{(\delta,\rho)}(k,v) = L_1(v) + L_2(k;\rho)$. 
\end{proof}

Finally we can conclude the proof of the main theorem. 

\begin{prop}
We have $aD(a) \to 0$ as $a\to \infty$. 
\end{prop}

\begin{proof}
The asymptotic estimates from the previous section (Proposition \ref{asymptotics}) imply that 
\begin{gather*}
a (g_a - \delta) \to 0, \quad \text{in } C^0(\an)\\
a (u_a - \rho) \to \frac{\la b + \overline X,y\ra}{\vert y\vert^3}, \quad \text{in } W^{1,p}(\an). 
\end{gather*} 
Also we have $D(a) = \mathcal D(g_a,u_a)$ and the Frechet differentiability implies that 
\[
\mathcal D(g_a,u_a) = \mathcal D(\delta,\rho) + L(g_a -\delta,u_a-\rho) + e(g_a - \delta, u_a -\rho)
\]
where 
\[
\frac{\vert e(k,v)\vert }{\|k\|_{C^0} + \|v\|_{W^{1,p}}} \to 0 
\]
as $\|k\|_{C^0} + \|v\|_{W^{1,p}}\to 0$. 
Note that $\mathcal D(\delta,\rho) = 0$. Hence multiplying by $a$ gives 
\[
aD(a) = L(a(g_a-\delta), a(u_a-\rho)) +  a e(g_a-\delta,u_a-\rho). 
\]
We have 
\begin{align*} 
ae(g_a-\delta,u_a-\rho) = a(\|g_a - \delta\|_{C^0} + \|u_a-\rho\|_{W^{1,p}}) \frac{e(g_a-\delta,u_a-\rho)}{\|g_a - \delta\|_{C^0} + \|u_a-\rho\|_{W^{1,p}}}.
\end{align*} 
Since $a(g_a-\delta) \to 0$ and $a(u_a-\rho)$ converges, we see that 
\[
a(\|g_a-\delta\|_{C^0} + \|u_a-\rho\|_{W^{1,p}})
\]
is bounded. 
Sending $a$ to infinity, this gives 
\[
\lim_{a\to\infty} aD(a) = L\left(0, \frac{\la b + \overline X,y\ra}{\vert y\vert^3}\right). 
\]
It remains to compute the right hand side. 

Define the function  
\[
v(y) = \frac{\la b + \overline X,y\ra}{\vert y\vert^3}.
\]
Then, using $\rho = \vert y\vert^{-1}$, we calculate that 
\begin{align*}
L(0,v) &= \int_{\an} \varphi'(\rho) \vert \grad \rho\vert^3 v \, dy + 3\int_{\an} \varphi(\rho) \vert \grad \rho\vert \la \grad \rho,\grad v\ra\, dy \\
&= \int_{\an} \frac{\varphi'(\vert y\vert^{-1})}{\vert y\vert^6} v(y)\, dy - 3\int_{\an} \frac{\varphi(\vert y\vert^{-1}) }{\vert y\vert^4}  \left\la \frac{y}{\vert y\vert}, \grad v(y)\right\ra \, dy.
\end{align*}
The co-area formula implies that 
\begin{align*}
\int_{\an} \frac{\varphi'(\vert y\vert^{-1})}{\vert y\vert^6} v(y)\, dy &= \int_1^4 \frac{\varphi'(r)}{r^9} \int_{\vert y\vert = r} \la b+\overline X,y\ra \, da(y)\, dr = 0. 
\end{align*} 
Likewise, the co-area formula also gives  
\begin{align*}
\int_{\an} \frac{\varphi(\vert y\vert^{-1})}{\vert y\vert^4} \left\la \frac{y}{\vert y\vert}, \grad v(y)\right\ra \, dy &= -2\int_1^4 \frac{\varphi(r)}{r^8} \int_{\vert y\vert = r} \la b+\overline X,y\ra \, da(y)\, dr = 0. 
\end{align*}
It follows that $L(0,v) = 0$ and this completes the proof. 
\end{proof}

\appendix

\section{The \texorpdfstring{$F$}{F} Functional} 
\label{F-Appendix} 

In this appendix, we summarize some results about the $F$ function introduced by Agostiniani, Mazzieri, and Oronzio in \cite{agostiniani2024green}. Let $g$ be a complete metric on $\R^3$ and assume that there exists a positive Green's function $u$ for $\lap_g$ with a pole at the origin and $u\to 0$ at infinity.  Further assume that $u$ is scaled so that the flux over each level set of $u$ is $4\pi$.  Note that in this set up every regular level set of $u$ is necessarily connected. 

\begin{defn}
For regular values $t > 0$ of $u$, the $F$ function is defined by 
\[
F(t) = 4\pi t - t^2 \int_{\{u=\frac 1 t\}} H \vert \grad u\vert\, da + t^3 \int_{\{u=\frac 1 t\}} \vert \grad u\vert^2\, da,
\]
where all the geometric quantities are computed with respect to the $g$ metric. 
\end{defn}

\begin{rem}
The function $u$ above corresponds to the function $1 - u$ in the original notation of \cite{agostiniani2024green}. We hope this will not cause confusion. Also, our convention is that 
\[
H = -\div\left(\frac{\grad u}{\vert \grad u\vert}\right)
\]
so that $H$ is positive when $u = \vert x\vert^{-1}$ on Euclidean space. 
\end{rem} 

One has the following monotonicity theorem for the $F$ function \cite{agostiniani2024green}.  

\begin{prop}
The function $F$ is absolutely continuous with 
\[
F'(t) = 4\pi + \int_{\{u=\frac 1 t\}} -\frac{R^{\Sigma_t}}{2} + \frac{\vert \grad^{\Sigma_t} \vert \grad u\vert \vert^2}{\vert \grad u\vert^2} + \frac{R}{2} + \frac{\vert \mathring A\vert^2}{2} + \frac{3}{4}\left(\frac{2\vert \grad u\vert}{u}-H\right)^2\, da
\]
for almost every $t$. Here we have abbreviated $\Sigma_t = \{u = \frac 1 t\}$. 
\end{prop} 

This has the following corollaries. 

\begin{corollary}
Assume that $g$ has non-negative scalar curvature.  Then $F(t)$ is a non-decreasing function of $t$. 
\end{corollary}

\begin{proof}
This follows from the formula for $F'(t)$ together with the Gauss-Bonnet theorem since the level sets of $u$ are connected. 
\end{proof}

\begin{corollary}
Assume that $g$ has non-negative scalar curvature. Then $F$ is non-negative. 
\end{corollary}

\begin{proof}
The asymptotics of the Green's function near the pole imply that
\[
\lim_{t\to 0^{-}} F(t) = 0.
\]
Then the result follows from the monotonicity of $F$. 
\end{proof}

We also note that the term involving the gradient squared is absolutely continuous on its own. 

\begin{prop}
\label{ac-gradient}
    The function $t\mapsto \int_{\{u=\frac 1 t\}} \vert \grad u\vert^2\, da$ is absolutely continuous with 
    \[
\frac{d}{dt} \int_{\{u=\frac 1 t\}} \vert \grad u\vert^2\, da = -t^{-2}\int_{\{u=\frac 1 t\}} H \vert \grad u\vert\, da
\]
    almost everywhere. 
\end{prop}

Lastly, we recall the following rigidity theorem for the $F$ function. 

\begin{prop}
\label{F-rigidity} 
Assume that $g$ has non-negative scalar curvature. If $F(t) \equiv 0$ then $g$ is flat. 
\end{prop}

\clearpage

\bibliographystyle{plain}
\bibliography{bibliography}

\end{document}